\documentclass[]{amsart}

\def\bn{\mathbb{N}}
\def\br{\mathbb{R}} 
\def\bc{\mathbb{C}} 
\def\h{\mathcal{H}} 
\def\k{\mathcal{K}}

\newcommand{\com}[1]{#1^{\prime}}
\newcommand{\inner}[2]{\langle #1,#2\rangle} 
\newcommand{\sumj}{\sum_{j=1}^{\infty}}
\newcommand{\hg}{\stackrel{h}{\otimes}}
\newcommand{\ehg}{\stackrel{eh}{\otimes}}
\newcommand{\shg}{\stackrel{\sigma}{\otimes}}
\newcommand{\bh}{{\rm B}(\mathcal{H})}
\newcommand{\bk}{{\rm B}(\mathcal{K})}
\newcommand{\bkh}{{\rm B}(\mathcal{K},\mathcal{H})}
\newcommand{\trh}{{\rm T}(\mathcal{H})}
\newcommand{\trk}{{\rm T}(\mathcal{K})}
\newcommand{\kh}{{\rm K}(\mathcal{H})}
\newcommand{\hsh}{{\rm C^2}(\mathcal{H})}
\newcommand{\tr}{{Tr\,}}
\newcommand{\fixa}{{\mathcal F}_a}
\newcommand{\cb}[1]{{\rm CB}(#1)}
\newcommand{\weakc}[1]{\overline{#1}}

\newtheorem{theorem}{Theorem}[section]
\newtheorem{lemma}[theorem]{Lemma}
\newtheorem{co}[theorem]{Corollary}
\newtheorem{pr}[theorem]{Proposition}
\theoremstyle{remark}
\newtheorem{re}[theorem]{Remark}
\theoremstyle{definition}

\newtheorem{ex}[theorem]{Example}
\numberwithin{equation}{section}

\begin{document}

\title[]{Fixed points of normal completely positive maps on $\bh$} 
\author{Bojan Magajna} 
\address{Department of Mathematics\\ University of Ljubljana\\
Jadranska 21\\ Ljubljana 1000\\ Slovenia}
\email{Bojan.Magajna@fmf.uni-lj.si}

\thanks{Acknowledgment. I am grateful to Milan Hladnik for his corrections of the first
draft of the paper and his suggestion concerning Example \ref{ex4}. I am also grateful to Victor Shulman 
for the discussion concerning Example \ref{ex4}.}




\begin{abstract}Given a sequence of bounded operators $a_j$ on a Hilbert space $\h$ with $\sumj a_j^*a_j=1=\sumj a_ja_j^*$,
we study the map $\Psi$ defined on $\bh$ by $\Psi(x)=\sumj a_j^*xa_j$ and its restriction $\Phi$ to the Hilbert-Schmidt class
$\hsh$. In the case when the sum $\sumj a_j^*a_j$ is norm-convergent we show in particular that the operator
$\Phi-1$ is not invertible if and only if the C$^*$-algebra $A$ generated by $\{a_j\}_{j=1}^{\infty}$ has an amenable trace. This is
used to show that $\Psi$ may have fixed points in $\bh$ which are not in the commutant $\com{A}$ of $A$
even in the case when the weak* closure of $A$ is injective. However, if $A$ is abelian, then all fixed points of $\Psi$ are in
$\com{A}$ even if the operators $a_j$ are not positive.
\end{abstract}

\maketitle

\section{Introduction and notation}

It is well known that all normal (= weak* continuous) completely positive maps on $\bh$ (the algebra of all bounded operators
on a separable Hilbert space $\h$) are of the form
\begin{equation}\label{1}\Psi_a(x)=\sumj a_j^*xa_j=a^*x^{(\infty)}a,\end{equation}
where $a_j\in\bh$ are such that the column $a:=(a_j)$ represents a bounded operator from $\h$ to $\h^{\infty}$, and
$x^{(\infty)}$ denotes the  block-diagonal operator matrix with $x$ along the diagonal.
The sum $a^*a=\sumj a_j^*a_j$ is convergent in the strong (weak, weak*,...) operator topology. 
If $a^*a=1$ (the identity operator on $\h$),
then the map $\Psi_a$ is unital. $\Psi_a$ is dual to the map $\Psi_{*a}$ defined on the trace class
$\trh$  by 
\begin{equation}\label{2}\Psi_{*a}(t)=\sumj a_jta_j^*.\end{equation}
So, if we assume in addition that the sum $\sumj a_ja_j^*$ is convergent in the strong 
operator topology, then the map $\Psi_a$ itself preserves $\trh$. If moreover $\sumj a_ja_j^*=1$, then the map $\Psi_a|\trh$ preserves the
trace (that is, $\tr(\Psi_a(t))=\tr(t)$ for all $t\in\trh$). Such maps are called unital quantum channels in quantum computation theory \cite{K}.
A selfadjoint operator $x\in\bh$ which is fixed by $\Psi_a$ (that is, $\Psi_a(x)=x$) represents a physical quantity that passes unchanged
through the quantum channel, so it is important to know the set $\fixa$ of all fixed points of $\Psi_a$. The structure of the set $\fixa$ is studied
in several papers (see e.g. \cite{AGG}, \cite{BJKW},  \cite{LZ},
\cite{Pr}, \cite{WJ} and references there.) Obviously $\fixa$ is a unital weak operator  closed
self-adjoint subspace of $\bh$ (in particular, it is spanned by positive elements) and $\fixa$ contains the
commutant $\com{A}$ of the C$^*$-algebra $A$ generated by the operators $a_j$. If it happens that the positive part $\fixa^+$ of 
$\fixa$ is closed under the 
operation $x\mapsto x^2$, then it is well known that $\fixa=\com{A}$ (\cite{BJKW}, 
\cite{AGG}). (For a proof, just note that $(ax-x^{(\infty)}a)^*(ax-x^{(\infty)}a)=
x^2+\Psi_a(x^2)-\Psi_a(x)x-x\Psi_a(x)=0$, hence $ax=x^{(\infty)}a$. 
The assumption that $\sumj a_ja_j^*=1$ is not needed for this conclusion.) 
It is proved in \cite{AGG} that each $x\in\fixa^+$ which can be diagonalized and the sequence of eigenvalues arranged in a decreasing order
is in fact in $\com{A}$. But in general 
$\fixa$ is not equal to $\com{A}$. Namely, since the map
$\Psi_a$ is a complete contraction and $\Psi_a|\com{A}$ is the identity,  $\Psi_a$ must be the identity also on the injective
envelope of $\com{A}$. Hence, if $\com{A}$ is not injective then $\fixa\ne\com{A}$. It is proved in \cite{AGG} that $\fixa$ is always an injective
operator space. If all $a_j$ are positive operators the operator $\Psi_a$ is called a (generalized) L\" uders operator. In \cite{AGG} an example
of L\" uders operator is given where $\fixa\ne\com{A}$. It is asked in \cite{AGG} if the injectivity of $\com{A}$ (or equivalently, the
injectivity of the weak* closure $\weakc{A}$ of $A$) implies the equality $\fixa=\com{A}$ 
for L\" uders operators. In physics the von Neumann algebras 
usually appear as direct limits of finite dimensional C$^*$-algebras and are 
therefore injective, so the question seems interesting also from the viewpoint of physics. 
We shall show that the answer is negative even in the case when $A$ is an irreducible subalgebra of $\bh$ (so that $\com{A}=\bc 1$)
and only finitely many $a_j$'s are non-zero. (We remark that without positivity requirement $a_j\geq0$ the question is much easier, one can construct
counterexamples by using direct sums of suitable Toeplitz operators.)

The basic idea of counterexample is to consider the action of L\" uders operators 
(where for simplicity we assume that only finitely many $a_j$'s are
nonzero or at least that the sum in (\ref{1}) is norm convergent) on the quotient $\bh/K$, where $K$ is a twosided ideal
in $\bh$. We will exploit the fact that the commutant $\dot{A}^c$ of the image $\dot{A}$ of $A$ in $\bh/K$ can be very large so that not all of its elements  can  be lifted to 
$\com{A}$. For example, if $K$ is the (unique) closed ideal $\kh$ of all compact operators, 
then it is a well-known consequence of Voiculescu's theorem \cite{Co2}
that $\dot{A}^c$ is so large that $\dot{A}^{cc}=\dot{A}$ (note that $A$ is separable), while $\com{A}$ consists of scalars only if $A$ is irreducible.
Now let $\dot{\Psi}$ be the map induced by $\Psi:=\Psi_a$ on $\bh/K$ and let $\dot{x}\in \dot{A}^c$ be such that $\dot{x}$ can not be lifted to an element in $\com{A}$.
Then $\dot{\Psi}(\dot{x})=\dot{x}$, hence, denoting by $x$ any lift in $\bh$ of $\dot{x}$, 
\begin{equation}\label{3}y:=\Psi(x)-x\in K.\end{equation}
Since $\dot{x}$ can not be lifted to $\com{A}$, it follows that $x+z\notin\com{A}$ for all $z\in K$. So, if we can find $z\in K$ such that $\Psi(x+z)=x+z$,
then we will have $x+z\in\fixa\setminus\com{A}$. Using (\ref{3}), the condition for $z$ is that 
$$(1-\Psi)(z)=y.$$
We could then find such a $z$ if we knew that the map $(1-\Psi)|K$ is invertible. But in the case $K=\kh$ the operator $(1-\Psi)|K$ can not be invertible
since its second adjoint on $\bh$ (the bidual of $\kh$) is just $1-\Psi$, which has 
nontrivial kernel (containing $\com{A}$). Similarly $(1-\Psi)|\trh$ is
not invertible. So we have to consider
other (non-closed) ideals, the simplest of which is the Hilbert--Schmidt class $\hsh$. But
in these case every operator that commutes with a C$^*$-algebra $A$ modulo $\hsh$ is a
perturbation of an element of $\com{A}$ by an element of $\hsh$ (see \cite{Hoo}).
So we will have to consider operators that commute with all $a_j$ modulo $\hsh$, but do
not commute module $\hsh$ with the whole C$^*$-algebra $A$ generated by the operators $a_j$.
(This is possible since the space $\hsh$ is not closed in the usual operator norm.)

In Section 2 we shall see that an operator $\Psi_a$ of the form (\ref{1}) (with the sums $\sumj a_j^*a_j=1=\sumj a_ja_j^*$ weak* converging) 
always preserves $\hsh$, so we may consider the restriction $\Phi_a:=\Psi_a|\hsh$. We shall prove that if the operator $\Phi_a-1$ is not invertible then
there exists a state $\rho$ on $\bh$ such that 
$$\rho(\sumj b_ja_j)=\rho(\sumj a_jb_j)$$
for all operators $b_j\in\bh$ such that the two series $\sumj b_jb_j^*$ and 
$\sumj b_j^*b_j$ are weak* convergent. Conversely, if there exists
a state $\rho$ on $\bh$ such that $\rho(cd)=\rho(dc)$ for all $d\in\bh$ and all $c$ in 
the C$^*$-algebra $A$ generated by $\{a_j\}_{j=1}^{\infty}\cup\{1\}$  and 
$$\sumj\rho(a_j^*a_j)=1,$$
then the map $\Phi_a-1$ is not invertible. Thus, in the case when the series $\sumj a_j^*a_j$ is norm convergent, $\Phi_a-1$ is not invertible
if and only if $A$ has an amenable trace in the sense of \cite{B}, \cite{BO}. This
result is then used in Section 3 to study fixed points of $\Psi_a$ on $\bh$.

In the beginning of Section 4 we will present some
general observations on the spectra of maps on $\bh$ of the form $\Theta:x\mapsto\sumj a_jxb_j$, where
$(a_j)$ and $(b_j)$ are two commutative sequences of normal operators such that
the sums $\sum a_ja_j^*$ and $\sum b_j^*b_j$ are weak* convergent. We observe that the spectrum of
$\Theta$ in the Banach algebra $\cb{\bh}$ of all completely bounded maps on $\bh$ is the same as the
the spectrum of $\Theta$ in certain natural subalgebras of $\cb{\bh}$. (Here some facts from the theory 
of operator spaces will be needed, but these results are not used in the rest of the paper.) The spectrum
of such a map can be much larger than the closure of the set $\sigma$ of all sums
$\sumj\phi(a_j)\psi(b_j)$, where $\phi¢$ and $\psi$ are characters on the 
C$^*$-algebras generated by $(a_j)$ and $(b_j)$, respectively, but all eigenvalues of
$\Theta$ are contained in $\sigma$. 

At the end of Section 4 we will provide a short proof of the fact that if the C$^*$-algebra
$A$ generated by the operators $(a_j)$ is abelian, then the fixed points of $\Phi_a$ are
contained in $\com{A}$.  For positive
operators $a_j$ this was proved in \cite{WJ} and  also in \cite{LZ}, but
our proof is different even in this case.

\section{Amenable traces and the spectrum of $\Phi_a$}

Throughout the section $a=(a_j)$ is a bounded operator from a separable Hilbert space
$\h$ to the direct sum $\h^{\infty}$ of countably many copies of $\h$,  
such that the components $a_j\in\bh$ satisfy
\begin{equation}\label{11}a^*a=\sumj a_j^*a_j=1=\sumj a_ja_j^*.\end{equation}
(The first equality is by the definition of $a$.) As in the Introduction,  $\Psi=\Psi_a$  denotes
the map on $\bh$ defined by
\begin{equation}\label{12}\Psi_a(x)=\sumj a_j^*xa_j=a^*x^{(\infty)}a.\end{equation} By $\hsh$ we denote the ideal of all Hilbert--Schmidt
operators on $\h$, and $\|x\|_2$ denotes the Hilbert--Schmidt norm of an element $x\in\hsh$,
which is defined by $\|x\|_2=\sqrt{\tr(x^*x)}$.

\begin{pr}\label{pr1} (i) $\Psi(\hsh)\subseteq\hsh$ and the restriction  
$\Phi:=\Psi|\hsh$ is a contraction, that is $\|\Phi(x)\|_2\leq\|x\|_2$ for all $x\in\hsh$.

(ii) For all $x\in\hsh$ the inequalities $$\|ax-x^{(\infty)}a\|_2^2\leq 
2\|x-\Phi(x)\|_2\|x\|_2\ \ \mbox{and}\ \ \|\Phi(x)-x\|_2
\leq\|ax-x^{(\infty)}a\|_2$$
hold.

(iii) The operator $\Phi-1$ is not invertible if and only if there exists a sequence
of selfadjoint elements $x_k\in\hsh$ with $\|x_k\|_2=1$ such that
$$\lim_{k\to\infty}\|\Phi(x_k)-x_k\|_2=0.$$
\end{pr}

\begin{proof}(i) Since $\|a\|=1$, we have that $aa^*\leq1$ (the identity operator on $\h^{\infty}$).
Using this and the equality $\sumj a_ja_j^*=1$,  we compute that for each $x\in\hsh$
\begin{align*}\|\Phi(x)\|_2^2=\tr\left(a^*x^{*(\infty)}aa^*x^{(\infty)}a\right)
\leq\tr\left(a^*(x^*x)^{(\infty)}a\right)\\=
\sumj\tr(a_j^*x^*xa_j)=\sumj\tr(xa_ja_j^*x^*)=\tr(xx^*)=\|x\|_2^2.\end{align*}

(ii) Using the relations $a^*a=1$, $aa^*\leq1$, $\tr(\Phi(x^*x))=\tr(x^*x)$ and the well-known properties of the trace
we have \begin{align*}\|ax-x^{(\infty)}a\|_2^2=\tr\left((ax-x^{(\infty)}a)^*(ax-x^{(\infty)}a)
\right)\\=\tr\left(x^*x+\Phi(x^*x)-\Phi(x)^*x-x^*\Phi(x)\right)\\
=\tr\left((x-\Phi(x))^*x+x^*(x-\Phi(x))\right)\\
\leq2\|x\|_2\|x-\Phi(x)\|_2.
\end{align*}
Similarly
$$\|\Phi(x)-x\|_2=\|a^*(x^{(\infty)}a-ax)\|_2\leq\|a^*\|\|x^{(\infty)}a-ax\|_2=\|ax-x^{(\infty)}a\|_2.$$

(iii) The existence of a sequence $(x_k)$ as in (iii) clearly implies that the map 
$\Phi-1$ is not invertible (in ${\rm B}(\bh)$). Conversely, if $\Phi-1$ is not invertible,
then $1$ is a boundary point of the spectrum of $\Phi$ since $\|\Phi\|\leq1$. But all boundary
points of the spectrum are approximate eigenvalues (\cite{Co}, p. 215), so there exists a sequence
of elements $x_k\in\hsh$ such that $\|x_2\|_2=1$ and $\lim\|\Phi(x_k)-x_k\|_2=0$.
By passing to an appropriate subsequence of real or imaginary parts of $x_k$ and normalizing
we can obtain a sequence of selfadjoint elements in $\hsh$ satisfying the condition in (iii).
\end{proof}

In the proof of the main result of this section we will need two simple facts
stated in the following remark.

\begin{re}\label{re1} If $x=(x_j)$ and $y=(y_j)$ are two operators from $\h$ to $\h^{\infty}$ of 
the Hilbert--Schmidt class (so that in particular $x_j, y_j\in\hsh$) then:

(i) $\|x\|_2=\|x^T\|_2$, where $x^T$ is the row $[x_j]$ regarded as the Hilbert--Schmidt
operator from $\h^{\infty}$ to $\h$.

(ii) $\tr(x^*y)=\sumj\tr(x_j^*y_j)$, where the series converges absolutely. 

Part (i) is immediate. To prove (ii), we choose an orthonormal basis $(\xi_k)$ of
$\h$ and compute that
$$\tr(x^*y)=\sum_{k=1}^{\infty}\sum_{j=1}^{\infty}\inner{x_j^*y_j\xi_k}{\xi_k}=
\sum_{j=1}^{\infty}\sum_{k=1}^{\infty}\inner{x_j^*y_j\xi_k}{\xi_k}=\sumj\tr(x_j^*y_j),$$
where the change of the order of summation is permissible since
$$\sum_{j,k=1}^{\infty}|\inner{x_j^*y_j\xi_k}{\xi_k}|\leq(\sum_{j,k=1}^{\infty}\|y_j\xi_k\|^2)^{1/2}
(\sum_{j,k=1}^{\infty}\|x_j\xi_k\|^2)^{1/2}=\|x\|_2\|y\|_2<\infty.$$
\end{re}

Recall  that a trace on a C$^*$-subalgebra $A\subseteq\bh$ is called
amenable if it can be extended to a state $\rho$ on $\bh$ such that $\rho(cd)=\rho(dc)$
for all $c\in A$ and $d\in\bh$ \cite{B}, \cite{BO}. We also recall the Powers-St\" ormer
inequality: $\|x-y\|_2^2\leq\|x^2-y^2\|_1$ for all positive $x,y\in{\rm C^2}(\h)$.
(A proof can be found for example in \cite{B}. Usually the inequality is used in the form
$\|xu-ux\|_2^2\leq\|x^2u-ux^2\|_1$ for positive $x\in\hsh$ and a unitary $u$.)

\begin{theorem}\label{th2}Let $A$ be the C$^*$-algebra generated by the identity and the operators
$a_j\in\bh$ satisfying (\ref{11}) and let $\Phi=\Phi_a$ be the restriction to $\hsh$ of the map
$\Psi$ defined by (\ref{12}). If $\Phi-1$ is not invertible then 
there exists a state $\rho$ on $\bh$ such that
\begin{equation}\label{13}\rho(b^Ta)=\rho(a^Tb)\ \mbox{for all}\ b=(b_j)\in{\rm B}(\h,\h^{\infty})\ 
\mbox{such that}\ b^T\in{\rm B}(\h^{\infty},\h).\end{equation}

Conversely, if $\rho$ is a state on $\bh$ such that $\rho(cd)=\rho(dc)$ for all $c\in A$ and
$d\in\bh$ and
\begin{equation}\label{14}\sumj\rho(a_j^*a_j)=1,\end{equation}
then the map $\Phi-1$ is not invertible.

Thus, if at least one of the series in (\ref{11}) is norm convergent, then the map $\Phi-1$
is not invertible if and only if $A$ has an amenable trace.
\end{theorem}

\begin{proof}If $\Phi-1$ is not invertible then by Proposition \ref{pr1} there exists a
sequence of selfadjoint elements $x_k$ in $\hsh$ with $\|x_k\|_2=1$ and 
$$\lim_{k\to\infty}\|ax_k-x_k^{(\infty)}a\|_2=0.$$ 
Let $\rho_k$ be the state on $\bh$ defined by $\rho_k(d)=\tr(dx_k^2)$ and let $\rho$ be
a weak* limit point of the sequence $(\rho_k)$. Note that for each $x\in\hsh$ and 
$b=(b_j)\in{\rm B}(\h,\h^{\infty})$ we have $\tr(a^Tbx^2)=\tr(xa^Tbx)$ and (by Remark
\ref{re1}(ii) since $ax$ and $(x^*b_j)$ are in ${\rm C}^2(\h,\h^{\infty})$)
$$\tr(b^Tx^{(\infty)}ax)=\sumj\tr(b_jxa_jx)=\sumj\tr(a_jxb_jx)=\tr(a^Tx^{(\infty)}bx).$$
Using this and Remark \ref{re1}(i) we now compute that
\begin{align*}|\tr(b^Tax_k^2)-\tr(a^Tbx_k^2)|\\
=\left|\tr\left(b^T(ax_k-x_k^{(\infty)}a)x_k\right)+
\tr\left((a^Tx_k^{(\infty)}-x_ka^T)bx_k\right)\right|\\
\leq\|b^T\|\|ax_k-x_k^{(\infty)}a\|_2+\|a^Tx_k^{(\infty)}-x_ka^T\|_2\|b\|\\
=(\|b\|+\|b^T\|)\|ax_k-x_k^{(\infty)}a\|_2\stackrel{k\to\infty}{\longrightarrow}0.
\end{align*}
Since $\rho$ is a weak* limit point of $(\rho_k)$, this implies that 
$\rho(b^Ta)=\rho(a^Tb)$. In particular $\rho(a_jd)=\rho(da_j)$ for all $a_j$ and all
$d\in\bh$, which implies that $\rho|A$ is an amenable trace.

Suppose now conversely, that $\rho$ is a state on $\bh$ satisfying (\ref{14}) and 
$\rho(cd)=\rho(dc)$ for all $c\in A$ and $d\in\bh$. Since the series in (\ref{14})
is convergent, given $\varepsilon>0$, there exists $m\in\bn$ such that
\begin{equation}\label{15}\sum_{j=m+1}^{\infty}\rho(a_ja_j^*)=\sum_{j=m+1}^{\infty}
\rho(a_j^*a_j)<\frac{\varepsilon}{8}.
\end{equation}
Since normal states are weak* dense in the state space of $\bh$, there exists a net
of positive operators $y_k\in\trh$ with the trace norm $\|y_k\|_1=1$ such that the
states $\rho_k(d):=\tr(dy_k)$ ($d\in\bh$) weak* converge to $\rho$. By passing to a
subnet we may assume that
\begin{equation}\label{16}|(\rho_k-\rho)\left(\sum_{j=m+1}^{\infty}(a_ja_j^*+a_j^*a_j)\right)
<\frac{\varepsilon}{4}.
\end{equation}
Let $a_{(m)}=(a_1,\ldots,a_m)\in{\rm B}(\h,\h^m)$. Observe that the trace class 
operators $a_{(m)}y_k-y_k^{(m)}a_{(m)}\in{\rm T}(\h,\h^m)$ converge weakly to $0$ 
since for all $d=[d_1,\ldots,d_m]\in{\rm B}(\h^m,\h)$ we have (denoting by $y^{(m)}$
the direct sum of $m$ copies of an operator $y$) 
\begin{align*}\tr(d(a_{(m)}y_k-y_k^{(m)}a_{(m)}))=\sum_{j=1}^m\tr(d_j(a_jy_k-y_ka_j))\\
=\sum_{j=1}^m\tr((d_ja_j-a_jd_j)y_k)\stackrel{k}{\longrightarrow}\sum_{j=1}^m\rho(d_ja_j
-a_jd_j)=0.
\end{align*}
Therefore suitable convex combinations of  operators $a_{(m)}y_k-y_k^{(m)}a_{(m)}$ must converge to $0$ in norm; 
thus, replacing the $y_k$'s by suitable convex combinations, we may assume that
$$\|a_{(m)}y_k-y_k^{(m)}a_{(m)}\|_1\stackrel{k}{\longrightarrow}0.$$
Let $x_k=y_k^{1/2}$. It follows from the Powers--St\" ormer inequality (by expressing the
components $a_j$ of $a_{(m)}$ as linear combinations of unitaries) that 
\begin{equation}\label{17}\sum_{j=1}^m\|a_jx_k-x_ka_j\|_2\stackrel{k}{\longrightarrow}0.
\end{equation}
Now we can estimate
\begin{align*}\|ax_k-x_k^{(\infty)}a\|_2^2=\sum_{j=1}^{m}\|a_jx_k-x_ka_j\|_2^2+
\sum_{j=m+1}^{\infty}\|a_jx_k-x_ka_j\|_2^2\\
\leq\sum_{j=1}^m\|a_jx_k-x_ka_j\|_2^2+2\sum_{j=m+1}^{\infty}(\|a_jx_k\|_2^2+\|x_ka_j\|_2^2)\\
=\sum_{j=1}^m\|a_jx_k-x_ka_j\|_2^2+2\sum_{j=m+1}^{\infty}(\tr(a_jx_k^2a_j^*+a_j^*x_k^2a_j))\\
=\sum_{j=1}^m\|a_jx_k-x_ka_j\|_2^2+2\rho_k(\sum_{j=m+1}^{\infty}a_j^*a_j+a_ja_j^*).
\end{align*}
Using  (\ref{15}) and (\ref{16}) it follows now that
$$\|ax_k-x_k^{(\infty)}a\|_2^2<\sum_{j=1}^m\|a_jx_k-x_ka_j\|_2^2+\varepsilon,$$
hence (\ref{17}) implies that $\|ax_k-x_k^{(\infty)}a\|_2^2<\varepsilon$ for some $k$.
Since $\varepsilon>0$ was arbitrary, Proposition \ref{pr1} tells us that the map
$\Phi-1$ is not invertible.

If the series $\sumj a_j^*a_j$ is norm convergent to $1$ then the condition (\ref{14})
is automatically satisfied for any state $\rho$. If $\rho|A$ is tracial then the same
conclusion holds if we assume the norm convergence of the series $\sumj a_ja_j^*=1$. Finally,
observe that for any state $\rho$ on $\bh$ satisfying $\rho(cd)=\rho(dc)$ for all
$c\in A$ and $d\in\bh$ the condition (\ref{14}) implies (\ref{13}) since
\begin{align*}|\rho(\sum_{j=m}^{\infty}b_ja_j)|\leq\sum_{j=m}^{\infty}|\rho(b_ja_j|
\leq\sum_{j=m}^{\infty}\rho(b_jb_j^*)^{1/2}\rho(a_j^*a_j)^{1/2}\\
\leq\|b\|(\sum_{j=m}^{\infty}\rho(a_j^*a_j))^{1/2}\stackrel{m\to\infty}{\longrightarrow}0\end{align*}
and similarly $|\rho(\sum_{j=m}^{\infty}a_jb_j)|\stackrel{m\to\infty}{\longrightarrow}0$.
\end{proof}

\begin{co}\label{co12}If the von Neumann algebra $\overline{A}$ generated by the operators $a_j$
(satisfying (\ref{11})) is finite and injective then the operator $\Phi_a-1$ is not
invertible. 
\end{co}

\begin{proof}Let $E:\bh\to\overline{A}$ be a conditional expectation, $\tau$ any normal
tracial state on $\overline{A}$ and $\rho=\tau E$. The state  $\rho=\tau E$ satisfies the 
condition (\ref{14}) and $\rho(cd)=\rho(dc)$ for all $c\in\weakc{A}$ and $d\in\bh$
(since $E$ is an $\weakc{A}$-bimodule map), hence the map $\Phi_a-1$ is not invertible.
\end{proof}

Given an arbitrary von Neumann algebra $R\subseteq\bh$, Haagerup \cite{Ha} proved that if 
the norm of every elementary operator on $\hsh$ of
the form $x\mapsto\sum_{j=1}^nu_jxu_j^*$, where the coefficients $u_j\in R$ are unitary,
is equal to $1$, then $R$ is finite and injective. The author does not know if the same
conclusion holds under the assumption that the norm of elementary operators of
the form $x\mapsto\sum_{j=1}^na_jxa_j$ is equal to $1$ for all positive $a_j\in R$.

\begin{re}We have seen in the proof of Theorem \ref{th2} that the condition (\ref{14})
implies (\ref{13}). But the two conditions are not equivalent. Indeed, let $R\subseteq\bh$ be
an abelian infinite dimensional von Neumann algebra, $\omega$ any non-normal state on $R$ and
$E:\bh\to R$ a conditional expectation. Since $\omega$ is not normal there exists in $R$
a sequence $(a_j)$ of mutually
orthogonal projections with the sum $1$ such that $\sumj\omega(a_j)<1$. Let $\rho=\omega E$,
a state on $\bh$. Even though $E$ is not necessarily 
weak* continuous the equalities
$$E(b^Ta)=\sumj E(b_j)a_j=\sumj a_jE(b_j)=E(a^Tb)$$
hold for all $b=(b_j)$ ($b_j\in\bh$)
such that the two sums $\sumj b_j^*b_j$ and $\sumj b_jb_j^*$ are weak* convergent. 
(This is so because $E$ is a completely positive $\overline{A}$-bimodule map; see
\cite{ER1} or \cite{HW}.)  Hence
$\rho(b^Ta)=\omega(E(b^Ta))=\omega(E(a^Tb))
=\rho(a^Tb)$. But $\sumj\rho(a_j^*a_j)=\sumj\omega(a_j)<1$.
\end{re}

We show now by an example that the condition (\ref{13}) is not automatically fulfilled 
by states satisfying $\rho(cd)=\rho(dc)$ for all $c\in A$ and $d\in\bh$.

\begin{ex}Choose an orthonormal basis $(\xi_j)$ ($j=1,2\ldots$) of $\h$ and let $a_j$ be the rank $1$ 
orthogonal projection onto $\bc\xi_j$. Then the C$^*$-algebra $A$, generated by $(a_j)$
and $1$, is the C$^*$-algebra of all convergent sequences acting as diagonal operators.
For each $j$ let $b_j$ be a rank $1$ partial isometry such that $b_jb_j^*=a_{2j}$ and
$b_j^*b_j=a_j$. Then
$$a^Tb=\sumj a_jb_j=0,$$
while
$$b^Ta=\sumj b_ja_j=\sumj b_j=:v$$
is an isometry with the range projection $p=vv^*=\sumj a_{2j}$ of infinite rank. Let
$q:\bh\to\bh/\kh=C(H)$ be the quotient map, $\theta$ a state on $C(H)$ such that 
$\theta(q(v))\ne0$, and $\rho:=\theta q$. Then $q(c)$ is a scalar for each $c\in A$,
hence for each $d\in\bh$
$$\rho(cd)=\theta(q(c)q(d))=\theta(q(c))\theta(q(d))=\rho(dc).$$
But nevertheless $\rho(b^Ta)=\rho(v)\ne0=\rho(a^Tb)$.
\end{ex}

\medskip

{\bf Problem.} Is the necessary condition (\ref{13}) also sufficient
for the conclusion of Theorem \ref{th2}? In other words, may the stronger 
condition (\ref{14}) be replaced by (\ref{13})? 

\medskip 
The answer is affirmative at least when $a=(a_j)$ is such that 
the operator $x^{(\infty)}a\in{\rm B}(\h,\h^{\infty})$ is of trace class for a dense
set of trace class operators $x\in\trh$. Namely, in this case we
can modify the proof of Theorem \ref{th2} as follows. First we approximate the state $\rho$ 
in Theorem \ref{th2} by normal states coming from
operators $y_k\in\trh$ such that the operators $y_k^{(\infty)}a\in {\rm B}(\h,\h^{\infty})$ are
of trace class. Then we verify that the sequence $(y_k^{(\infty)}a-ay_k)$ converges weakly 
to $0$. Finally we  show that 
$\|\sqrt{y_k}^{(\infty)}a-a\sqrt{y_k}\|_2\stackrel{k\to\infty}{\longrightarrow}0$.
For the last step wee need the following consequence of the Powers-St\" ormer inequality.

\begin{pr}For all operators $b\in\bkh$ and positive operators $x\in\trh$, $y\in\trk$
the inequality
\begin{equation}\label{PS}\|by-xb\|_2^2\leq \gamma\|by^2-x^2b\|_1\|b\|\end{equation}
holds, where $\gamma=\frac{8}{9}\sqrt{3}$.
\end{pr}

\begin{proof}By considering the operator
$$\left[\begin{array}{ll}
0&b\\
b^*&0\end{array}\right]$$
instead of $b$ and $$\left[\begin{array}{ll}
x&0\\
0&y\end{array}\right]$$
instead of both $x$ and $y$, the proof can be reduced immediately to the case when $b=b^*$
and $y=x$. (Further, in this case we may replace $b$ by $b+s1$ for a suitable scalar $s$ so that we may
assume that both $\|b\|$ and $-\|b\|$ are in the spectrum of $b$.) Denote $\beta=\|b\|$
and for  $t\in\br\setminus\{0\}$ let
$$u_t=(b-ti)(b+ti)^{-1},\ \ \ \mbox{so that}\ \ \ b=ti(1+u_t)(1-u_t)^{-1}.$$
Since $u_t$ is unitary, we have by the Powers-St\" ormer inequality
$$\|u_tx-xu_t\|_2^2\leq\|u_tx^2-x^2u_t\|_1,$$
which can be rewritten as
\begin{equation}\label{18}\|(b+ti)^{-1}z_t(b+ti)^{-1}\|_2^2\leq\|2ti(b+ti)^{-1}(bx^2-x^2b)
(b+ti)^{-1}\|_1,\end{equation}
where $z_t:=2ti(bx-xb)$. Since $\|z_t\|_2\leq\|b+ti\|^2\|(b+ti)^{-1}z_t(b+ti)^{-1}\|_2$,
(\ref{18}) implies that
$$\|z_t\|_2^2\leq2t\|b+ti\|^4\|((b+ti)^{-1}\|^2\|bx^2-x^2b\|_1.$$
Thus (since $\|(b+ti)^{-1}\|^2\leq t^{-2}$ and $\|b+ti\|^2\leq\beta^2+t^2$)
$$\|bx-xb\|_2^2\leq\frac{(\beta^2+t^2)^2}{2t^3}\|bx^2-x^2b\|_1.$$
Taking the minimum over $t$ of the right-hand side of this inequality, we obtain the
desired estimate (\ref{PS}).
\end{proof}

\section{On the fixed points of the map $\Psi_a$}

As we indicated already in the Introduction, Theorem \ref{th2} implies the following corollary.

\begin{co}\label{co21}With the notation as in Theorem \ref{th2}, suppose that  
the C$^*$-algebra $A$ has no amenable traces and that the two series $\sumj a_j^*a_j=1=
\sumj a_ja_j^*$ are norm convergent. If there exists an operator $y\in\bh$ such that the
operator $y^{(\infty)}a-ay$ is in the Hilbert-Schmidt class and $y$ is not in 
$\com{A}+\hsh$, then the operator $\Psi=\Psi_a$
defined on $\bh$ by $\Psi_a(x)=\sumj a_j^*xa_j$ has fixed points which are
not in $\com{A}$.
\end{co}

\begin{proof}Observe that $y-\Psi(y)\in\hsh$ since 
$$y-\Psi(y)=a^*(ay-y^{(\infty)}a)$$
and $y^{(\infty)}a-ya$ is in the Hilbert-Schmidt class by the hypothesis. By Theorem \ref{th2}
the map $(\Psi-1)|\hsh$ is invertible, hence there exists a $z\in\hsh$ such that
$(\Psi-1)(z)=y-\Psi(y)$. This means that $\Psi(y+z)=y+z$. Hence $x:=y+z$ is a fixed point
of $\Psi$, and $x$ is not in $\com{A}$ since $y\notin\com{A}+\hsh$ and $z\in\hsh$.
\end{proof}

Now we give an example which satisfies the conditions of Corollary \ref{co21} and solves a 
problem left open in \cite{AGG}.

\begin{ex}Let $v_i$ ($i=1,2$) be the isometries defined on $\h=\ell^2(\bn)$ by
$$v_1e_j=e_{2j}\ \ \ \mbox{and}\ \ \ v_2e_j=e_{2j+1}\ \ \ (j=0,1,2,\ldots),$$
where $(e_j)$ is an orthonormal basis of $\h$.
Then $v_1v_1^*+v_2v_2^*=1$ and the C$^*$-algebra $A$ generated by $\{v_1,v_2\}$ is the Cuntz 
algebra $O(2)$ (defined in \cite{Cu} or \cite{EK}), which has no tracial states (and is nuclear). 

To show that $A$ is irreducible, choose any $d\in\com{A}$ and let 
$$de_0=\sum_{j=0}^{\infty}\alpha_je_j\ \ \ (\alpha_j\in\bc).$$
Then 
$$\sum_{j=0}^{\infty}\alpha_je_j=de_0=dv_1^*e_0=v_1^*de_0=
\sum_{j=0}^{\infty}\alpha_{2j}e_j,$$
which implies that $\alpha_{j}=\alpha_{2j}$ for all $j$.
Similarly, from $0=dv_2^*e_0=v_2^*de_0=\sum_{j=0}^{\infty}\alpha_{2j+1}e_j$ we see that 
$\alpha_{2j+1}=0$ for all $j$. It follows that $\alpha_j=0$ for all $j>0$.
Thus $de_0=\alpha_0e_0$ and consequently $d(v_1^{k_1}v_2^{k_2}v_1^{k_3}\ldots)e_0=
(v_1^{k_1}v_2^{k_2}v_1^{k_3}\ldots)de_0=\alpha_0(v_1^{k_1}v_2^{k_2}\ldots)e_0$
for any sequence $k_1,k_2,\ldots$ in $\bn$. Since the linear span of vectors of the form 
$(v_1^{k_1}v_2^{k_2}\ldots)e_0$ is dense in $\h$, it follows that $d=\alpha_01$.

We will show that there exists a positive diagonal operator $y\in\bh$ such that 
$$yv_2=v_2y,\ \ yv_1-v_1y\in\hsh,\ \ \mbox{but}\ \ y\notin\bc1+\hsh=\com{A}+\hsh.$$
Let $ye_j=t_je_j$, where $t_j$ are nonnegative scalars to be specified. The condition
$yv_2=v_2y$ means that 
\begin{equation}\label{21}t_{2j+1}=t_j\ \ \ (j=0,1,2,\ldots).\end{equation}
On the other hand, the condition $yv_1-v_1y\in\hsh$ means that
\begin{equation}\label{22}\sum_{j=0}^{\infty}(t_{2j}-t_j)^2<\infty.\end{equation}
To satisfy these two conditions, choose $t_j$, for example, as follows. 
If $j$ is of the form $j=2^k$ ($k\in\bn$) let $t_j=(k+1)^{-1/2}$. If $j$ is not a power
of $2$ define $t_j$ recursively by
$$t_j=\left\{\begin{array}{ll}
t_{\frac{j}{2}},&\mbox{if $j$ is even};\\
t_{\frac{j-1}{2}},&\mbox{if $j$ is odd}.\end{array}\right.$$
Then $t_{2j+1}=t_j$ for all $j\in\bn$, so (\ref{21}) holds. Further, $t_{2j}=t_j$ 
for all $j$ which are not
powers of $2$, hence the sum in (\ref{22}) reduces to
$$\sum_{k=0}^{\infty}(\frac{1}{\sqrt{k+1}}-\frac{1}{\sqrt{k+2}})^2<\infty.$$
The so defined operator $y$ is not in $\bc1+\hsh$ since the series
$\sum_{j=0}^{\infty}(t_j+\alpha)^2$ diverges for all $\alpha\in\bc$.

Finally, we write $v_1$ and $v_2$ as linear combinations of positive elements 
$a_j\in A$ ($j=1,\ldots,8$)  such that $\sum_{j=1}^8 a_j^2\leq 1$. Define 
$a_0=(1-\sum_{j=1}^8a_j^2)^{1/2}$
and $a=(a_0,\ldots,a_8)$. Then $\Psi_a$ is a L\" uders operator for which not all fixed 
points are in $\com{A}$ ($=\bc1$), since $y$ commutes modulo $\hsh$ with all $a_j$ and does
not commute with all $a_j$.
\end{ex}

\section{The case of commuting operators}

In this section we study the spectrum and fixed points of normal completely bounded
maps on $\bkh$, where $\h$ and $\k$ are separable Hilbert spaces. 
We denote by $\cb{\bkh}$ the space of all completely bounded maps on $\bkh$. 
Given  C$^*$-subalgebras $A\subseteq\bh$ and $B\subseteq\bk$, we let
$A\ehg B$ be the Banach subalgebra of  $\cb{\bkh}$ consisting
of all maps $\Theta$ that can be represented in the form
\begin{equation}\label{201}\Theta(x):=\sumj c_jxd_j,\end{equation}
where $c_j\in A$ and $d_j\in B$ are such that the row $c=[c_j]$ and the column $d=(d_j)$
represent bounded operators in ${\rm B}(\h^{\infty},\h)$ and ${\rm B}(\k,\k^{\infty})$, respectively.
Thus  the sums 
\begin{equation}\label{31}\sumj c_jc_j^*\ \ \ \ \mbox{and}\ \ \ \ \sumj d_j^*d_j\end{equation}
converge in the strong operator topology. We will write such a map $\Theta$ 
simply as 
$$\Theta=c\odot d=\sumj c_j\otimes d_j.$$
The space $A\ehg B$ coincides with the extended Haagerup tensor product (defined in 
\cite{BS}, \cite{ER3}, \cite{M1}), but we shall not need this fact. The subspace
$A\hg B$ of $A\ehg B$, consisting of elements $c\odot d\in A\ehg B$ 
for which the two sums in (\ref{31}) are norm convergent, is a Banach subalgebra of $A\ehg B$, and
can be identified with the Haagerup tensor product, but again we shall not need this last fact. 
If $M$ and $N$ are von Neumann algebras then $M\ehg N$ coincides with the space 
${\rm NCB}_{\com{M},\com{N}}(\bkh)$ of all normal completely bounded
$\com{M},\com{N}$-bimodule endomorphisms of $\bkh$ (see \cite{S} or \cite[1.2]{M2};
here $\com{M}$ denotes the commutant of $M$).
It is well known that a weak* continuous map $\Theta$ between Banach spaces is invertible
if and only if its preadjoint map $\Theta_*$ is invertible \cite{Co}. Thus, if $\Theta\in
M\ehg N$ is invertible, then so is $\Theta_*$ (as a bounded map on ${\rm T}(\h,\k)$), hence
$\Theta^{-1}=((\Theta_*)^{-1})^*$ is weak* continuous. Since $\Theta^{-1}$ is also an
$\com{M},\com{N}$-bimodule map, it follows that $M\ehg N$ is an inverse-closed subalgebra
of $\cb{\bkh}$. The spectrum of an element $c$ in a Banach algebra $A$ is denoted by $\sigma_A(c)$.
We summarize the above discussion in the following proposition.

\begin{pr}\label{le}If $M$ and $N$ are von Neumann subalgebras of $\bh$ and 
$\bk$ (respectively) then $$\sigma_{M\ehg N}(\Theta)=\sigma_{\cb{\bkh}}(\Theta)$$
for each $\Theta\in M\ehg N$. 
\end{pr}

In many cases the above Proposition can be sharpened to the identity
$(M\ehg N)^{cc}=M\ehg N$. Namely, it is known (see \cite{ER1} or \cite{HW})
that the commutant $(M\ehg N)^c$
of $M\ehg N$ inside $\cb{\bkh}$ is the algebra ${\rm CB}_{M,N}(\bkh)$ of all completely 
bounded $M,N$-bimodule endomorphisms of $\bkh$, which we will
denote simply by $\com{M}\shg\com{N}$,  thus
\begin{equation}\label{101}(M\ehg N)^c=\com{M}\shg\com{N}.\end{equation}
(We remark that the notation $\com{M}\shg\com{N}$ usually means the normal Haagerup 
tensor product as defined in \cite{EKi}, \cite{ER3}, \cite[p. 41]{BLM},
but the two algebras $\com{M}\shg\com{N}$ and ${\rm CB}_{M,N}(\bkh)$ are naturally completely isometrically and weak* homeomorphically
isomorphic by \cite{EKi} (a simpler proof of a more general fact is in \cite[4.4]{M3}).)
By a surprising result of Hoffmeier and Wittstock \cite{HW} the commutant of $\com{M}\shg\com{N}$ in $\cb{\bkh}$ 
consists only of weak* continuous maps, if
$M$ and $N$ do not have central parts of type $I_{\infty,n}$ for $n\in\bn$, that is
\begin{equation}\label{102}(\com{M}\shg\com{N})^c=M\ehg N.\end{equation}
(In \cite{HW} only the case $N=M$ is considered, but the usual argument with the direct sum $M\oplus N$ reduces the general
situation to this case.) This holds in particular when $M$ and $N$ are abelian, thus, in this 
case we deduce from (\ref{101}) and (\ref{102}) that $(M\ehg n)^{cc}=M\ehg N$.

For noncommuting sequences $(c_j)$ and $(d_j)$ not much is known about the spectrum of
the operator $\Theta=c\odot d$ defined by (\ref{201}). For example, if $d_j=c_j$ are positive, it is
not known even if the spectrum of $\Theta$ is contained in $\br^+$ \cite{N}. We mention
here the following consequence of results of Shulman and Turovskii \cite{Sh},
which improves \cite[Corollary 6]{N}.

\begin{pr}Suppose that $c_j\in\bh$, $d_j\in\bk$ are positive and such that 
$\sumj\|c_j\|\|d_j\|<\infty$. If for each $j$ at least one of the operators $c_j$, $d_j$ is compact
then all eigenvalues of the operator $\Theta=c\odot d$ defined by (\ref{201}) on $\bkh$ are in $\br^+$.
\end{pr}

\begin{proof}By \cite[6.10]{Sh} each eigenvector corresponding to a nonzero eigenvalue
$\lambda$ of $\Theta$ is nuclear, hence in particular in the 
Hilbert-Schmidt class ${\rm C^2}(\k,\h)$. Since the restriction
$\Theta|{\rm C^2}(\k,\h)$ is a positive operator on a Hilbert space, its spectrum is contained
in $\br^+$, hence $\lambda\in\br^+$.
\end{proof}

We denote by $\Delta(A)$ the spectrum (that is, the space of all multiplicative linear
functionals) of a commutative Banach algebra $A$. If $A$ and $B$ are commutative operator
algebras  then it is easy to see that 
\begin{equation}\label{32}\Delta(A\hg B)=\Delta(A)\times\Delta(B).\end{equation}
For the spectrum of $A\ehg B$, however, there is no such simple formula. In the case
when $M$ and $N$ are (abelian) von Neumann algebras there is an injective contraction
from $M\ehg N$ into $M\overline{\otimes}N$ (which will be regarded as inclusion and is dual to 
the natural contraction $M_{*}
\stackrel{\wedge}{\otimes} N_{*}\to M_{*}\hg N_{*}$ \cite[1.5.13]{BLM}, \cite[6.1]{ER3}), and one might conjecture
that the spectrum of an element  of $M\ehg N$ is the same as the spectrum of its
image in $M\overline{\otimes}N$, but this is not always true even in the special case
$M=\ell^{\infty}(\bn)=N$. In this case $C:=M\overline{\otimes} N$ is the von Neumann algebra 
$\ell^{\infty}(\bn\times\bn)$ of
all bounded sequences  on $\bn\times\bn$. Further, $D:=M\ehg N$ is the algebra of all
Schur multipliers on ${\rm B}(\ell^2(\bn))$ (see \cite{Pi1}, Theorem 5.1), which consists of all
sequences $d\in\ell^{\infty}(\bn\times\bn)$ such that the double sequence $[d_{i,j}x_{i,j}]$ is a
matrix of a bounded operator on $\ell^2(\bn)$ for every $[x_{i,j}]$ representing a bounded operator
on $\ell^2(\bn)$. Such an element $d\in D$ is invertible in $C$ if and only if the closure
of the set $\{d_{i,j}\}$ in $\bc$ does not contain $0$, but this does not guarantee invertibility of
$d$ in $D$. To see this, we consider the following example suggested to us by Milan Hladnik and Victor
Shulman.

\begin{ex}\label{ex4}Let $D_0$ be the subalgebra of $D$ consisting of Toeplitz-Schur multipliers, that is,
Schur multipliers $d=[d_{i-j}]$ that are constant along the diagonals. If $d$ is invertible in $D$, then
$d^{-1}$ is also the inverse of $d$ in $C$, hence $d^{-1}$ consists of the double sequence 
$[d_{i-j}^{-1}]$, which is 
in $D_0$; so $D_0$ is inverse-closed in $D$. On the other hand, it is known that 
the entries of each Toeplitz-Schur multiplier
$[d_{i-j}]$ are the Fourier coefficients of a complex regular Borel measure $\mu$ on 
the unit circle $\mathbb{T}$ (that is, 
$d_k=\int_{\mathbb{T}} \overline{z}^k\, d\mu$) and conversely; that is, $D_0$ is isomorphic to the measure algebra
$M(\mathbb{T})$ for the convolution. (A proof of this can be found in \cite{AP}.) 
But by \cite[5.3.4]{R} there exists
a noninvertible measure $\mu\in M(\mathbb{T})$ such that the Fourier coefficients of $\mu$ 
are all real and $\geq1$, so the corresponding
Schur multiplier is invertible in $C$ but not in $D$. Moreover, by \cite[Theorem 6.4.1]{R}  the spectrum of such
a multiplier $[d_{i-j}]$ can contain any point in $\bc$ even if $d_k\in[-1,1]$ for all $k$.
\end{ex}

We remark that the spectra of elementary operators  
$x\mapsto\sum_{j=1}^mc_jxd_j$, where $m$ is finite and $c=(c_j), d=(d_j)\subseteq\bh$
are two commutative families, have been intensively studied in the past (see \cite{Cur} 
and the references in \cite{Cur} and in \cite{AM}), but the
results do not apply to the case of infinite $m$, where the two series $\sumj c_jc_j^*$ and $d_j^*d_j$ converge in the weak* topology.
Even if we assume that all the components $c_j$ and $d_j$ are normal operators, the above example suggests that the spectrum of
$c\odot d$ can not be described in terms of spectra of $c_j$ and $d_j$ in the same way as for finite $m$-tuples. 

If $A$ is an abelian Banach algebra and $c=(c_j)$ is a sequence of elements
in $A$ we set
$$\sigma_A(c)=\{(\rho(c_1),\rho(c_2),\ldots):\, \rho\in\Delta(A)\}.$$

\begin{lemma}\label{le31}If $c=(c_j)$ is a sequence in a commutative unital C$^*$ algebra $A
\subseteq\bh$ such that the series $\sumj c_j^*c_j$ is norm convergent,
then $\sigma_A(c)$ is a norm compact subset of $\ell^2$. If this sum is merely weak* 
convergent, then
$\sigma_{A}(c)$ is a weakly compact subset of $\ell^2$.
\end{lemma}

\begin{proof}For any character $\rho\in \Delta(A)$ and any finite $n$ we have 
$$\sum_{j=1}^n|\rho(c_j)|^2=\rho(\sum_{j=1}^nc_j^*c_j)\leq\|c\|^2,$$
which implies that $(\rho(c_j))\in\ell^2$ with $\|(\rho(c_j))\|\leq\|c\|$. 
It is easy to prove that the map $\rho\mapsto(\rho(c_j))$ from $\Delta(A)$ to $\ell^2$
is weak* to weak continuous, so its range $\sigma_A(c)$ is a weakly compact set since $\Delta(A)$ is 
weak* compact. 
If the series $\sumj c_j^*c_j$ is norm convergent, then the same map is weak* to norm
continuous, hence $\sigma_A(c)$ is a norm compact set in this case.
\end{proof}

Given two elements $\lambda=(\lambda_j)$ and $\mu=(\mu_j)$ in $\ell_2$ we denote
$$\lambda\cdot\mu:=\sumj\lambda_j\mu_j.$$ Further, for two subsets $\sigma_j\subseteq\ell^2$, we denote
$$\sigma_1\cdot\sigma_2:=\{\lambda\cdot\mu:\, \lambda\in\sigma_1,\ \mu\in\sigma_2\}.$$
Since the map $(\lambda,\mu)\mapsto\lambda\cdot\mu$ is continuous, $\sigma_1\cdot\sigma_2$ is a compact
subset of $\bc$ if $\sigma_1$ and $\sigma_2$ are norm compact subsets of $\ell^2$.

\begin{pr}\label{pr32}Let $(c_j)$ and $(d_j)$ be two commutative families of normal operators in $\bh$ and $\bk$ (respectively)
such that the two series (\ref{31}) 
are weak* convergent. Let A and B be the C$^*$ algebras generated by $\{1\}\cup(c_j)$ and $\{1\}\cup(d_j)$, respectively, and $\weakc{A}$, $\weakc{B}$
their weak* closures, so that the map $\Theta=c\odot d$ is an element of $\weakc{A}\ehg\weakc{B}$.

(i) If the two series (\ref{31}) are norm convergent (that is, if $\Theta\in A\hg B$) then 
$\sigma_{\cb{\bh}}(\Theta)=\sigma_{A}(c)\cdot\sigma_{B}(d)$.

(ii) In general the point spectrum of $\Theta$ is contained in $\sigma_{A}(c)\cdot\sigma_{B}(d)$.
\end{pr}

\begin{proof}The spectrum of an element $\Theta$ in a unital commutative Banach algebra $D$ is always equal
to $\{\rho(\Theta):\, \rho\in\Delta(D)\}$. This applies to our element $\Theta=c\odot d$ in $D=\weakc{A}\ehg\weakc{B}$. 
Given $\rho\in\Delta(D)$, denote  $\phi=\rho|(\weakc{A}\otimes1)$ and $\psi=\rho|(1\otimes{B})$. Then $\phi\in\Delta(A)$,
$\psi\in\Delta(B)$ and (by the norm continuity) $\rho|(\weakc{A}\hg\weakc{B})=\phi\otimes\psi$. Conversely,  any two characters
$\phi\in\Delta(\weakc{A})$ and $\psi\in\Delta(\weakc{B})$ define the character $\phi\otimes\psi$ on $D$ by $(\phi\otimes\psi)(\sumj x_j\otimes y_j)
=\sumj\phi(x_j)\psi(y_j)$. So, if the two series (\ref{31}) are norm convergent then 
$\sigma_{D}(\Theta)=\sigma_{\weakc{A}}(c)\cdot\sigma_{\weakc{B}}(d).$  Since all 
characters on 
$A$ and $B$ extend to characters on $\weakc{A}$
and $\weakc{B}$, respectively, it also follows that $\sigma_{A\hg B}(\Theta)=\sigma_A(c)\cdot\sigma_B(d)=\sigma_D(\Theta)$. 
By Proposition \ref{le} we have that $\sigma_{\cb{\bkh}}(\Theta)=\sigma_D(\Theta)$ for each
$\Theta\in D$. This concludes the proof of (i). 

To prove (ii), let $\lambda$ be an eigenvalue of $\Theta$ and $x\in\bkh$ a corresponding 
nonzero eigenvector, so that $\Theta(x)=\lambda x$. 
By a variant of Egoroff´s theorem 
\cite[p. 85]{T} 
in any neighborhoods of the identity $1$ (in the strong operator
topology) there exists projections $e\in\weakc{A}$ and $f\in\weakc{B}$ such that the two series
$\sumj c_jc_j^*e$ and $\sumj d_j^*d_jf$ converge uniformly. We may choose $e$ and $f$ so that $exf\ne0$. Since
$$(ec)(exf)^{\infty}(df)=\lambda exf,$$
$\lambda$ is an eigenvalue of $(ec)\odot (df)$, hence 
$\lambda\in\sigma_{\weakc{A}}(ec)\cdot\sigma_{\weakc{B}}(df)$ by (i). 
Since $\phi(e)\in\{0,1\}$ for each $\phi\in\Delta(\weakc{A})$ and similarly for $f$, it follows that $\lambda\in\sigma_{\weakc{A}}(c)\cdot
\sigma_{\weakc{B}}(d)=\sigma_A(c)\cdot\sigma_B(d)$ if $\lambda\ne0$. If $\lambda=0$, we apply the result just obtained to the map $\Theta+1=\tilde{c}\odot\tilde{d}$, where
$\tilde{c}=[1,c_1,c_2,\ldots]$ and $\tilde{d}=(1,d_1,d_2,\ldots)$ and the eigenvalue $1$ of this map.
\end{proof}

\begin{theorem}\label{th40}Let $a=(a_j)$ and $b=(b_j)$ be two commutative sequences of normal operators on (separable) Hilbert
spaces $\h$ and $\k$ (respectively) such that $\sumj a_ja_j^*=1$ and $\sumj b_j^*b_j=1$, 
where the sums are weak* convergent. Then the fixed points
of the map $\Theta=a\odot b=\sumj a_j\otimes b_j$ on $\bh$ are precisely the operators 
$x\in\bh$ that intertwine $a$ and
$b^*$ (that is, $a_jx=xb_j^*$ for all $j$).
\end{theorem}

\begin{proof}Clearly the intertwiners of $a$ and $b^*$ are fixed points of $\Theta$ since 
$\sumj b_j^*b_j=1$, so only the converse
needs a proof. By considering 
$$\left[\begin{array}{ll}
a_j&0\\
0&b_j^*\end{array}\right],\ \ \left[\begin{array}{ll}
a_j^*&0\\
0&b_j\end{array}\right]\ \ \mbox{and}\ \ \left[\begin{array}{ll}
0&x\\
0&0\end{array}\right]$$
instead of $a_j$, $b_j$ and $x$ (respectively), the proof can easily be reduced to the case where $b_j=a_j^*$. So we assume
that $b_j=a_j^*$ for all $j$ and we have to prove that each fixed point $x$ of $\Theta$ 
commutes with all $a_j$. 
Let $A$ be the C$^*$-algebra generated by $1$ and $(a_j)$ and let $e$ be the spectral measure on $\Delta:=\Delta(A)$
such that 
$$c=\int_{\Delta}\hat{c}(\phi)\, de(\phi)$$
for all $c\in A$, where $\hat{c}$ is the Gelfand transform of $c$ \cite[p. 266]{Co}.  
It suffices to show that $xe(K)=e(K)x$ for each compact subset $K$ of $\Delta$ or, 
equivalently, that
$e(K)^{\perp}xe(K)=0$, where $e(K)^{\perp}=1-e(K)$. Since $e(K)^{\perp}=e(\Delta\setminus K)$ is the join
of all the projections $e(H)$ for compact subsets $H$ of $K^c:=\Delta\setminus K$, it suffices to show that $e(H)xe(K)=0$
for all such $H$. Assume the contrary, that
$$e(H)xe(K)\ne0$$
for some compact $H\subseteq K^c$. Consider the orthogonal decomposition 
\begin{equation}\label{103}\h=e(H)\h\oplus e(K)\h\oplus e(H^c\cap K^c)\h\end{equation} and
let $x=[x_{k,l}]$ be the corresponding representation of $x$ by a $3\times 3$ 
operator matrix. With respect to the decomposition
(\ref{103}) each operator $a_j$ is represented by a diagonal matrix 
$a_j=c_j\oplus d_j\oplus f_j$ (where, for example, $c_j=a_je(H)|e(H)\h$). Then the $(1,2)$ 
entry of the matrix 
$\Theta(x)=\sumj a_jxa_j^*$ is
$\sumj c_jx_{1,2}d_j^*$, where $x_{1,2}=e(H)xe(K)\ne0$. From $\Theta(x)=x$ we have
$$\sumj c_jx_{1,2}d_j^*=x_{1,2},$$
which means that $1$ is an eigenvalue of the map $\Theta_{c,d^*}:=\sumj c_j\otimes d_j^*$.
By Proposition \ref{pr32} 
\begin{equation}\label{104}1=\inner{\lambda}{\mu}\ \ \ \mbox{for some}\ 
\lambda\in\sigma_{Ae(H)}(c),\ \mu\in\sigma_{Ae(K)}(d).\end{equation}
Since $\sumj c_jc^{*}_j=e(H)$ and $\sumj d_jd_j^*=e(K)$, it follows that $\|\lambda\|\leq1$
and $\|\mu\|\leq1$, hence (\ref{104}) implies that $\mu=\lambda$. Therefore
\begin{equation}\label{105}\sigma_{Ae(H)}(c)\cap\sigma_{Ae(K)}(d)\ne\emptyset.\end{equation}

On the other hand, the map $\hat{a}:\phi\mapsto(\phi(a_1),\phi(a_2),\ldots)$ from 
$\Delta$ into
$\ell^2$ is injective.  Since the C$^*$-algebra $Ae(H)$
is isomorphic to $C(H)$ (complex valued continuous functions on $H$) by Tietze's theorem,
$\Delta(Ae(H))\cong H$. (That is, all characters of $Ae(H)$ are evaluations at points
of $H$.) Hence $\sigma_{Ae(H)}(c)=\sigma_{Ae(H)}(ae(H))=\hat{a}(H)$.
Similarly $\sigma_{Ae(K)}(d)=\hat{a}(K)$. Since $H$ and $K$ are disjoint and $\hat{a}$ is
injective, $\sigma_{Ae(H)}(c)$ and $\sigma_{Ae(K)}(d)$ must also be disjoint, but this 
is in contradiction with (\ref{105}).
\end{proof}

\medskip
{\bf Problem.} Does the conclusion of Theorem \ref{th40} still hold if, instead of commutativity,
we assume that each of the two sequences $(a_j)$ and $(b_j)$ is contained in a finite
von Neumann algebra?


\begin{thebibliography}{99}
 

\bibitem{AP} A. B. Aleksandrov and V. V. Peller, {Hankel and Toeplitz-Schur multipliers},
Math. Ann. {\bf 324} (2002), 277-327. 


\bibitem{AM} P. Ara and M. Mathieu, {\em Local multipliers of C$^*$-algebras,}
Springer Monographs in Math., Springer-Verlag, Berlin, 2003.
 

\bibitem{AGG} A. Arias, A. Gheondea, S. Gudder, {\em Fixed points of quantum operations}, 
J. Math. Phys. {\bf 43} (2002), no. 12, 5872–5881. 





\bibitem{BLM} D. P. Blecher and C. Le Merdy, {\em Operator algebras and their modules,}
L.M.S. Monographs, New Series {\bf 30}, Clarendon Press, Oxford, 2004.  

\bibitem{BS} D. P. Blecher and R. R.  Smith, {\em The dual of the Haagerup tensor product}, 
J. London Math. Soc. {\bf 45} (1992), 126–144.

\bibitem{BJKW} O. Bratteli, P. E. T. Jorgensen, A. Kishimoto and R. F. Werner,
{\em Pure states on O$_d$}, J. Operator Theory {\bf 43} (2000), 97--143.


\bibitem{B} N. P. Brown, {\em  Invariant means and finite representation theory of 
$C^*$-algebras}, Mem. Amer. Math. Soc. {\bf 184} (2006), no. 865.

\bibitem{BO} N. P. Brown and N. Ozawa, {\em C$^*$-algebras and finite dimensional approximations,}
GSM {\bf 88}, AMS, Providence, RI, 2008.

\bibitem{Co} J. B. Conway, {\em A course in functional analysis}, GTM {\bf 96}, Springer, Berlin, 1985.

\bibitem{Co2} J. B. Conway, {\em A course in operator theory}, GSM {\bf 21}, Amer. Math. 
Soc., Providence, RI, 2000. 

\bibitem{Cu} J. Cuntz, {\em Simple C*-Algebras generated by isometries}, 
Commun. math. Phys. {\bf 57} (1977), 173—185.




\bibitem{Cur} R. E. Curto, {\em Spectral theory of elementary operators}, 
Elementary Operators and Appl. (M. Mathieu editor),
World Scientific, Singapore, 1992.





\bibitem{EKi} E. G. Effros and A. Kishimoto, {\em Module maps and the Hoschild-Johnson cohomology}, Indiana.
Univ. Math. J. {\bf 36} (1987), 257--276.


\bibitem{ER1} E. G. Effros and Z. -J. Ruan, {\em Representations of operator bimodules and 
their applications}, J. Operator Theory {\bf 19} (1988), 137–-158.


\bibitem{ER3} E. G. Effros and Z. -J. Ruan, {\em Operator space tensor products and Hopf
convolution algebras}, J. Operator Theory {\bf 50} (2003), 131--156.

\bibitem{EK}  D. E. Evans and Y. Kawahigashi, {\em Quantum symmetries on operator algebras},
Oxford Mathematical Monographs, Oxford Univ. Press, New York, 1998. 



\bibitem{Ha} U. Haagerup, {\em Injectivity and decomposition of completely
bounded maps,} pp. 170--222, Lecture Notes in Math. {\bf 1132}, Springer-Verlag,
Berlin, 1985.

\bibitem{HW} H. Hofmeier and G. Wittstock, {\em A bicommutant theorem for completely bounded module
homomorphisms}, Math. Ann. {\bf 308} (1997), 141--154.
 
\bibitem{Hoo} T. B. Hoover, {\em Derivations, homomorphisms and operator ideals},
Proc. Amer. Math. Soc. {\bf 62} (1977),  293-298. 





\bibitem{K} D. W. Kribs, {\em A quantum computing primer for operator theorists}, Linear Algebra Appl. 400 (2005), 147–167.


\bibitem{LZ} L. Long and S. Zhang, {\em Fixed points of commutative super--operators}, J. Phys. A: Math. Theor. {\bf 44} (2011), 095201, 10pp.

\bibitem{M2} B. Magajna, {\em The Haagerup norm on the tensor product of operator modules}, J. Funct. Anal. {\bf 129} (1995), 325--348.

\bibitem{M1} B. Magajna, {\em Strong operator modules and the Haagerup tensor product}, Proc. London Math. Soc. {\bf 74} (1997), 201–240.


\bibitem{M3} B. Magajna, {\em Duality and normal parts of operator modules}, J. Funct. Anal. {\bf 219} (2005), 306--339.



\bibitem{N} G. Nagy, {\em On spectra of L\" uders operations}, J. Math. Phys. {\bf 49} (2008), no. 2, 022110, 8 pp.
  


\bibitem{Pi1} G. Pisier, {\em Similarity problems and completely bounded maps}, Lecture Notes in Math. 
{\bf 1618}, Springer, Berlin, 1996. 


\bibitem{Pr} B. Prunaru, {\em Toeplitz operators associated to commuting row contractions}, 
J. Funct. Anal. 254 (2008), no. 6, 1626–-1641.


\bibitem{R} W. Rudin, {\em Fourier analysis on groups}, Wiley, New York, 1990. 

\bibitem{Sh} V. S.  Shulman and Yu. V. Turovskii, {\em Topological radicals, II.
applications to spectral theory of multiplication operators}, Elementary Operators and
Appl. (R. E. Curto and M. Mathieu editors), Operator Th. Adv. Appl. {\bf 212}, 
Birkh\" auser, Basel, 2011. 

\bibitem{S}  R. R. Smith,  {\em Completely bounded module maps  and  the 
Haagerup tensor product}, J. Funct. Anal. {\bf 102} (1991), 156--175.

\bibitem{T}  M. Takesaki, {\em Theory of operator algebras I,} Springer-Verlag,
New-York, 1979.



\bibitem{WJ} L. Weihua and W. Junde, {\em Fixed points of commutative L\" uders operations}, 
J. Phys. A {\bf 43} (2010), no. 39, 395206, 9 pp.

\end{thebibliography}
\end{document}